\documentclass[a4paper,12pt]{amsart}

\usepackage[T1]{fontenc}
\usepackage[utf8]{inputenc}

\usepackage[english]{babel}

\usepackage{amsmath, amssymb, amsthm, amsfonts}
\usepackage{mathrsfs}
\usepackage{tikz}
\usetikzlibrary{arrows}
\usetikzlibrary{matrix}
\usetikzlibrary{cd}
\usepackage[left=2.5cm, right=2.5cm, top=2cm]{geometry}
\usepackage{enumerate}
\usepackage{hyperref}

\addto\extrasenglish{
	
}

\theoremstyle{plain}
\newtheorem{thm}{Theorem}
\newtheorem{lem}{Lemma}
\newtheorem{cor}{Corollary}

\newtheorem{prop}{Proposition}
\newtheorem*{thm*}{Theorem}

\theoremstyle{definition}
\newtheorem{defi}{Definition}
\newtheorem{rema}{Remark}
\newtheorem{exam}{Example}

\newtheorem*{ques*}{Question}
\newtheorem*{conss*}{Constructions}

\newcommand{\laur}{\left[t,t^{-1} \right]}
\newcommand{\polyp}{\left[ t \right]}
\newcommand{\polym}{\left[ t^{-1} \right]}

\newcommand{\powel}{\left[ \left[ t,t^{-1} \right] \right]}

\newcommand{\sub}{\textnormal{sub}}

\newcommand{\PPU}{\textnormal{PPU}}
\newcommand{\PU}{\textnormal{PU}}
\newcommand{\Hil}{\mathcal{H}}

\newcommand{\R}{\mathbb{R}}
\newcommand{\C}{\mathbb{C}}
\newcommand{\A}{\mathcal{A}}
\newcommand{\K}{\mathbb{K}}
\newcommand{\Z}{\mathbb{Z}}
\newcommand{\Hill}{\mathcal{H}\left[ \left[ t,t^{-1} \right] \right]}
\newcommand{\Hilp}{\mathcal{H}\left[ \left[ t \right] \right]}
\newcommand{\Hilm}{\mathcal{H}\left[ \left[ t^{-1} \right] \right]}

\begin{document}

\title{Right $\ell$-groups associated with von Neumann algebras}
\author{Carsten Dietzel}
\email{carstendietzel@gmx.de}
\address{Institute of algebra and number theory, University of Stuttgart, Pfaffenwaldring 57, 70569 Stuttgart, Germany}
\date{\today}

\begin{abstract}
In \cite{rump_goml}, Rump defined and characterized noncommutative universal groups $G(X)$ for generalized orthomodular lattices $X$.

We give an explicit construction of $G(X)$ in terms of \emph{pure paraunitary groups} when $X$ is the projection lattice of a von Neumann algebra. The results given here extend some of those in \cite{dietzel}.
\end{abstract}

\maketitle

\section*{Introduction}

For details on quantum logic, the interested reader can consult, for example, \cite{cohen}.

Hilbert spaces over $\K = \C$ provide a framework for quantum mechanics in that they contain all possible states a certain physical system can be in.

A very classical Hilbert space in quantum mechanics is the space $\Hil = L^2(\R)$ of square-integrable complex-valued functions on the real line whose elements represent the states of one-dimensional systems.

Closed subspaces of $\Hil$ represent certain physical properties of the observed system. A simple property in the case of $\Hil = L^2(\R)$, for example, is the property of being of \emph{positive parity} which mathematically amounts for a wavefunction $\psi$ to be even, i.e. $\psi(x) = \psi(-x)$ for all $x$. The statement \glqq$\psi$ is of positive parity\grqq\, is a prototypical example of a proposition in what is called \emph{quantum logic}.

Closed subspaces, representing quantum-mechanical \glqq propositions\grqq\,, provide one example of a semantical framework for quantum logic. This logic, however, is an example of a \emph{non-classical} logic. A vague explanation for this is that a physical system initially need not be in the state which is being measured but the proposition \glqq$\psi$ is of positive parity\grqq\ (for example) is a \emph{result} of the measuring.

A better explanation can be given as follows: classical logic is \emph{distributive} in the sense that if we are given propositions, say $P$, $Q$ and $R$, they fulfil the distributive laws, i.e.
\begin{align*}
P \wedge ( Q \vee R ) & \Leftrightarrow (P \wedge Q) \vee (P \wedge R), \\
P \vee ( Q \wedge R ) & \Leftrightarrow (P \vee Q) \wedge (P \vee R).
\end{align*}
The lattice of closed subspaces of a Hilbert space, however, is not distributive, in general. In fact when the Hilbert space is of infinite dimension it is not even modular, as is discussed in \cite[Problem 15]{halmos}.

Quantum logic is thus not classical. In fact, classical logic is \emph{contained} in quantum logic. We explain this further. A semantic framework for classical logic is provided by the \emph{Boolean algebras} but how does a semantic framework for quantum logic look like that is independent of the notion of Hilbert space? The answer has been obtained in a series of articles written by Birkhoff, von Neumann et al. (\cite{Birkhoff_von_Neumann}, \cite{neumann_algebraic_generalization}) and Husimi (\cite{Husimi}) - it is given by the \emph{orthomodular lattices}:

An orthomodular lattice is a bounded lattice $X$ (with its lowest and greatest elements denoted by $0$ and $1$) and an order-reversing involution $\ast: X \to X$ such that the following axioms hold:
\begin{align}
x \wedge x^{\ast} & = 0 \tag{OL1} \label{eq:ol1} \\
x \vee x^{\ast} & = 1  \tag{OL2} \label{eq:ol2} \\
x \leq y & \Longrightarrow x \vee (x^{\ast} \wedge y) = y \tag{OML} \label{eq:oml}
\end{align}

Typical examples of OMLs are
\begin{enumerate}[i)]
\item Boolean algebras, with $\ast$ being Boolean negation,
\item the lattice of closed subspaces of a Hilbert space, with $\ast$ being the operation of taking orthogonal complements or, more generally,
\item the lattice of closed $\A$-invariant subspaces of a Hilbert space $\Hil$, where $\A$ is a von Neumann algebra acting on $\Hil$, the $\ast$-operation being the same as above.
\end{enumerate}

There is a connection between OMLs and group theory which we will investigate in this article:

Let $X$ be an OML. A \emph{group-valued measure} (which we abbreviate by \emph{gvm}) on $X$ is a mapping $\mu: X \to G$ where $G$ is a group, such that $\mu(x \wedge y) = \mu(x)\mu(y)$ whenever $y \geq x^{\ast}$. It is clear that then there must also be a \emph{universal} gvm on $X$, i.e. a fixed gvm $\iota: X \to G(X)$ - where $G(X)$ is the \emph{structure group} of $X$ - such that the following universal property holds:

For any gvm $\mu: X \to G$, there is a unique homomorphism of groups $\overline{\mu}: G(X) \to G$ such that $\mu = \overline{\mu}\iota$.

Rump proved in \cite{rump_goml} that the submonoid $S(X)$ generated by $X$ in $G(X)$ is the negative cone of a right-invariant lattice order on $G(X)$. He furthermore gave a characterization of all possible structure groups in terms of ordered groups. This is one of many of Rump's results connecting logical algebras (\emph{L-Algebras}) with ordered algebraic structures. A classical precursor of this kind of results is Mundici's theorem on the equivalence of MV-algebras and abelian $\ell$-groups with strong order unit (\cite{mundici}).

However, due to the broad variety of OMLs, a general description of their structure groups $G(X)$ is likely to be general as well. In fact, a \glqq generic\grqq\, $G(X)$ is simply a group defined by generators and relations which are derived from the respective OML.

However, one gets more specific results when restricting to a more specific class of objects. This is essentially what we are doing in this article - namely, we give a rather concrete realization of $G(X)$ whenever $X$ is the OML of $\A$-invariant subspaces for a von Neumann-algebra $\A$ on a Hilbert space $\Hil$.

We will prove:

\begin{thm*}
Let $\A$ be a von Neumann algebra on a Hilbert space $\Hil$.

If $X(\A^{\prime})$ is the OML of $\A^{\prime}$-invariant subspaces where $\A^{\prime}$ is the commutant of $\A$, then there is an isomorphism of groups $G(X(\A^{\prime})) \cong \PPU(\A)$, where
\[
\PPU(\A) = \left\{ \sum_{i = -\infty}^{\infty} t^i \varphi_i \in \A\laur: \left( \sum_{i = -\infty}^{\infty}  t^i \varphi_i \right)\left( \sum_{i = -\infty}^{\infty}  t^{-i} \varphi_i^{\ast} \right) = 1 \wedge \sum_{i = -\infty}^{\infty} \varphi_i = 1 \right\}
\]
is the \emph{pure paraunitary group} of $\A$.
\end{thm*}

The case when $\A$ is a factor of type $I_n$ has essentially been covered by the author in the article \cite{dietzel} where a realization of $G(X)$ is given in the case when $X$ is the OML associated with a hermitian, anisotropic bilinear form on a finite-dimensional vector space (not necessarily over $\R$ or $\C$). Some arguments given in the present article will be similar to the arguments given in the preceeding one. However, we also tried to make some of them more streamlined in the von-Neumann-case and hope that we succeeded in doing so.

The outline of this article is as follows:

\autoref{sec:structure_groups} is rather a collection of definitions and facts: we define orthomodular lattices and explain how their structure groups are constructed. Furthermore, we give the definition of right $\ell$-groups and explain Rump's characterization theorem.

In \autoref{sec:paraunitary_groups}, we associate with each Banach-$\ast$-algebra $\A$ a so-called \emph{paraunitary group} $\PU(\A)$ and define a certain subgroup - the \emph{pure} paraunitary group $\PPU(\A)$. On these groups, we define a right-invariant partial order. To give the reader some intuition, we explicitly calculate the pure paraunitary groups of abelian $C^{\ast}$- and von Neumann-algebras.

\autoref{sec:ppu_is_lattice_ordered} is the heart of this article: we prove that the right-invariant order of $\PPU(\A)$ is a lattice when $\A$ is a von Neumann-algebra. We will essentially do this by establishing an order-isomorphism between $\PPU(\A)$ and a sublattice of the subspace lattice of a suitable Hilbert space. Thus, the property of the latter being a lattice (which is rather obvious) carries over to $\PPU(\A)$ (where it is less obvious).

In \autoref{sec:ppu_is_structure_group} we deduce that $\PPU(\A)$ is indeed the structure group of the OML $X(\A^{\prime})$. This will be done by showing that the right $\ell$-group $\PPU(\A)$ fulfils the conditions of Rump's characterization theorem.

In \autoref{sec:remarks}, we discuss how the pure paraunitary groups might be made \glqq more analytic\grqq .

\section{Structure groups of orthomodular lattices} \label{sec:structure_groups}

\subsection{Orthomodular lattices}

Let $X$ be a bounded lattice where we denote the lowest and greatest elements by $0$ resp. $1$.

We call an antitone involution $\ast: X \to X$ an \emph{orthocomplementation} if for any $x \in X$, we have the identities:
\begin{align*}
x \wedge x^{\ast} & = 0 \\
x \vee x^{\ast} & = 1.
\end{align*}

We can now introduce the topic of this subsection:

\begin{defi}
An \emph{orthomodular lattice} (\emph{OML}, for short) is a bounded lattice $X$ with orthocomplementation $\ast: X \to X$ such that the \emph{orthomodular law}
\begin{equation} \label{eq:orthomodular_law}
x \leq y \Rightarrow x \vee (x^{\ast} \wedge y) = y \tag{OML}
\end{equation}
holds.
\end{defi}

On each OML, one can define the (co-)orthogonality relations:
\begin{align*}
x \bot y & : \Leftrightarrow y \leq x^{\ast} \tag{$\bot$} \\
x \top y & : \Leftrightarrow y \geq x^{\ast}. \tag{$\top$}
\end{align*}
One can easily see that $\bot, \top$ are symmetric and are connected by the rule $x \top y \Leftrightarrow x^{\ast} \bot y^{\ast}$.

We define a \emph{partial monoid} as a set $X$ together with a distinguished element $e$ (its \emph{neutral} element) and a \emph{partial} mapping $\cdot: X \times X \rightharpoonup X$ such that:
\begin{enumerate}[i)]
\item for all $x \in X$, $e \cdot x$ and $x \cdot e$ are always defined, and we have $e \cdot x = x \cdot e = x$,
\item whenever either $(x \cdot y) \cdot z$ or $x \cdot (y \cdot z)$ is defined, so is the other, and in this case we have $(x \cdot y) \cdot z = x \cdot (y \cdot z)$ (\emph{partial associativity}).
\end{enumerate}

If the existence of $x \cdot y$ implies the existence of $y \cdot x$, and $x \cdot y = y \cdot x$ holds in this case, we say $X$ is \emph{partially commutative}.

Each OML is in fact a partial monoid:

\begin{prop} \label{prop:omls_are_partial_monoids}
Let $X$ be an OML. The partial operation
\[
x \oplus y =
\begin{cases} x \vee y & x \bot y \\
\text{not defined} & \text{else}  \end{cases}
\]
makes $X$ a partial monoid with $e_X = 0$.

Dually, the partial operation
\[
x \sqcap y =
\begin{cases} x \wedge y & x \top y \\
\text{not defined} & \text{else}  \end{cases}
\]
makes $X$ a partial monoid with $e_X = 1$.

Both partial monoid structures are partially commutative.
\end{prop}

\begin{proof}
We only prove the first statement for the second statement is proved similarly.

Partial commutativity follows immediately from the symmetry of $\bot$ and the commutativity of $\vee$.

For all $x$ we clearly have $0 \bot x$, and $0 \vee x = x$, so $0$ is indeed a neutral element for $\oplus$.

It remains to show that the existence of $(x \oplus y) \oplus z$ implies that of $x \oplus (y \oplus z)$. $\vee$ is associative, so this already suffices for partial associativity.

So, assume that $(x \oplus y) \oplus z$ exists. This means that $x \bot y$ and $z \bot (x \oplus y)$. Therefore $z \leq (x \vee y)^{\ast}\leq y^{\ast}$ which implies $z \bot y$ - this shows that $y \oplus z$ exists. Similarly, one sees that $z \bot x$.

From $x \bot y, z$ we deduce $x \leq (y^{\ast} \wedge z^{\ast}) = (y \vee z)^{\ast}$. This finally shows that $x \oplus (y \oplus z)$.

By using partial commutativity (or mimicking the argument from above), one can show that the existence of $x \oplus (y \oplus z)$ implies that of $(x \oplus y) \oplus z$, too.
\end{proof}

In this article we will be concerned about a certain class of OMLs derived from von Neumann algebras. We first recall some basic definitions and facts about von Neumann algebras (details can be found in \cite{Dixmier_Neumann}, for example):

Let $\Hil$ be some complex Hilbert space (with inner product denoted by $\left\langle -,- \right\rangle$ and norm denoted by $\Vert . \Vert$). Let $\mathcal{B}(\Hil)$ be the Banach-$\ast$-algebra of bounded operators on $\Hil$ where we define the $\ast$-operation, as usual, by taking adjoints, i.e. $\left\langle \varphi x, y \right\rangle = \left\langle x, \varphi^{\ast} y \right\rangle$.

We define the \emph{commutant} of some subset $M \subseteq \mathcal{B}(\Hil)$ by
\[
M^{\prime} := \left\{ \varphi \in \mathcal{B}(\Hil): \forall m \in M: \varphi m = m \varphi  \right\}.
\]
A \emph{von Neumann algebra} on $\Hil$ is a subalgebra $\A \subseteq \mathcal{B}(\Hil)$ with $\A^{\prime \prime} = \A$ that is closed under the $\ast$-operation. It is easily seen that $\A^{\prime}$ is then a von Neumann algebra, too.

We call $\pi \in \mathcal{B}(\Hil)$ a \emph{projection} when $\pi = \pi^{\ast}$ and $\pi^2 = \pi$. For a closed linear subspace $M \subseteq \Hil$, we denote by $\pi_M$ the (unique) projection of $\Hil$ onto $M$.

A closed linear subspace $M$ is called \emph{$\A$-invariant} whenever $\A M \subseteq M$.

The following duality between projections in $\A$ and $\A^{\prime}$-invariant subspaces will play a crucial role later:

\begin{prop} \label{prop:invariant_subspaces_are_annihilated}
A closed linear subspace $M \subseteq \Hil$ is $\A^{\prime}$-invariant if and only if the orthogonal projections $\pi_M, \pi_{M^{\ast}}$ lie in $\A$.
\end{prop}

\begin{proof}
We always have $1_\Hil \in \A$. Together with $\pi_{M^{\ast}} = 1_\Hil - \pi_M$ this shows that we only have to prove (resp. assume) that $\pi_M \in \A$.

First of all, let $M$ be $\A^{\prime}$-invariant.

$M^{\ast}$ then is $\A^{\prime}$-invariant as well: when $y \in \Hil$ fulfils $\left\langle x,y \right\rangle = 0$ for all $x \in M$ then for any $\varphi \in \A^{\prime}, x \in M$ we have $\left\langle x, \varphi y \right\rangle = \left\langle \varphi^{\ast} x, y \right\rangle = 0$ because $\varphi^{\ast} x \in M$, since $\varphi^{\ast} \in \A^{\prime}$ and $M$ is $\A^{\prime}$-invariant.

We must show that $\pi_M$ commutes with all $A \in \A^{\prime}$: let $A \in \A^{\prime}$. Writing $x \in \Hil$ as $x = x_M + x_{M^{\ast}}$ with $x_M \in M, x_{M^{\ast}}\in M^{\ast}$, we have
\[
\pi_M A x = \pi_M (Ax_M + \underbrace{Ax_{M^{\ast}}}_{\in M^{\ast}}) = Ax_M = A \pi_M x
\]
which proves that $\pi_M \in \A^{\prime\prime} = \A$.

On the other hand, let $\pi_M \in \A$. We show that $M$ is invariant under all $A \in \A^{\prime}$:

If $x \in M$, we have
\[
Ax = A \pi_M x = \pi_M A x
\]
from which we conclude that also $Ax \in M$.

\end{proof}

\begin{prop} \label{prop:invariant_subspaces_form_omls}
Let $\A$ be a von Neumann algebra acting on the Hilbert space $\Hil$. Then the lattice of $\A$-invariant closed subspaces becomes an OML under the operations:
\begin{align*}
M \vee N & = \overline{M + N} \\
M \wedge N & = M \cap N \\
M^{\ast} & = \left\{y \in \Hil: \forall x \in M: \left\langle x,y \right\rangle = 0 \right\}
\end{align*}
\end{prop}

\begin{defi} \label{def:XA}
We denote the OML described in \autoref{prop:invariant_subspaces_form_omls} by $X(\A)$.
\end{defi}

\begin{proof}[Proof of \autoref{prop:invariant_subspaces_form_omls}]
Intersections and closed sums of $\A$-invariant closed subspaces are again closed $\A$-invariant subspaces. This is proven by standard arguments. It is furthermore clear that $\overline{M + N}$ and $M \cap N$ are the least resp. greatest subspaces containing resp. contained in $M, N$.

It is also clear from the definition that $M \mapsto M^{\ast}$ is antitone as soon as we know that $M^{\ast}$ is closed and $\A$-invariant whenever $M$ is:

If $M$ is closed and $\A$-invariant, $M^{\ast}$ is clearly closed. Let $y \in M^{\ast}$ and $\varphi \in \A$, then for all $x \in M$ we have $\left\langle x, \varphi y \right\rangle = \left\langle \varphi^{\ast} x, y \right\rangle = 0$ because $\varphi^{\ast}x \in M$ due to $\A$ being $\ast$-closed and $M$ being $\A$-invariant.

It is well-known that $(M^{\ast})^{\ast}$ whenever $M \subseteq \Hil$ is closed. It is also known that $M^{\ast}$ complements $M$, i.e. $M \cap M^{\ast} = 0$ and $M \oplus M^{\ast} = \Hil$.

It remains to show \autoref{eq:orthomodular_law}:

Let $M \subseteq N$ be $\A$-invariant and closed. $N$ can be regarded as a Hilbert space in its own right and it is easily seen that $M^{\ast} \cap N$ is the orthogonal complement of $M$ within the Hilbert space $N$. Therefore, $M \oplus (M^{\ast} \cap N) = N$ which is exactly \autoref{eq:orthomodular_law} in this framework.
\end{proof}

\subsection{Right \texorpdfstring{$\ell$}{l}-groups}

As before, we first present the topic of this subsection:

\begin{defi}
A \emph{right-ordered group} is a group $G$, equipped with a partial order - denoted by $\leq$ - which is invariant under right-multiplication, meaning that the implication
\[
h_1 \leq h_2 \Rightarrow h_1g \leq h_2g
\]
is valid for all $g,h_1,h_2 \in G$.

If $G$ is right-ordered, we call it \emph{right lattice-ordered}, if $G$ becomes a lattice under $\leq$. We then say, $G$ is a \emph{right lattice-ordered group}, or, for short, a \emph{right $\ell$-group}.
\end{defi}

By a homomorphism of right-ordered groups we will mean a monotone homomorphism of groups. By a homomorphism of right $\ell$-groups we will mean a homomorphism of the underlying groups which also respects the lattice operations.

It is well-known that right-ordered groups can be described purely in algebraic terms. To give the description, we need the notion of a \emph{positive} resp. \emph{negative cone}:

If $G$ is right-ordered, we define its \emph{positive cone} as
\[
G^+ := \left\{ g \in G : g \geq e \right\}
\]
and, similarly, its \emph{negative cone} as
\[
G^- := \left\{ g \in G : g \leq e \right\}.
\]

Furthermore, for an arbitrary group $G$, we call a sub\emph{monoid} $P \subseteq G$ \emph{pure} if $P \cap P^{-1} = 1$ where $P^{-1} = \left\{g^{-1} : g \in P \right\}$. 

The following proposition is fundamental in the theory of right-ordered groups:

\begin{prop} \label{prop:pure_submonoids_are_cones}
Let $G$ be a group. For any right-order on $G$, the positive cone $G^+$ is a pure submonoid of $G$.

Vice versa, each pure submonoid $P \subseteq G$ gives rise to a right-order with $G^+ = P$ via the rule $g \geq h : \Leftrightarrow gh^{-1} \in P$.
\end{prop}

\begin{proof}
See \cite[Theorem 1.5.1.]{kopytov}.
\end{proof}

We can now present a class of right $\ell$-groups investigated by Rump in \cite{rump_goml}.

With a partial monoid we can always associate both a \emph{structure monoid} as a \emph{structure group}:

\begin{defi}
Let $X$ be a partial monoid with operation $\oplus$:

The \emph{structure monoid} $S(X)$ is a monoid, together with a mapping $\iota_S: X \to S(X)$ such that the following holds:
\begin{enumerate}[i)]
\item $\iota_S$ is a homomorphism of partial monoids, where $S(X)$ is seen as a partial monoid in the obvious way,
\item if $f: X \to M$ is such a homomorphism to another monoid $M$ there is a \emph{unique} homomorphism of monoids $\bar{f}: S(X) \to M$ such that $f = \iota_S \circ \bar{f}$.
\end{enumerate}

Analogously, the \emph{structure group} $G(X)$ is a group, together with a mapping $\iota_G: X \to G(X)$ such that the following holds:
\begin{enumerate}[i)]
\item $\iota_G$ is a homomorphism of partial monoids, where $G(X)$ is seen as a partial monoid in the obvious way,
\item if $f: X \to G$ is such a homomorphism to another group $G$ there is a \emph{unique} homomorphism of groups $\bar{f}: G(X) \to G$ such that $f = \iota_G \circ \bar{f}$.
\end{enumerate}
\end{defi}

It is easy to see that the structure monoids and groups of a partial monoid are unique and do always exist.

Uniqueness is in both cases a mere categorial fact. Furthermore, $S(X)$ can be constructed as the universal monoid generated by the nonzero elements of $X$ under the relations $x \oplus y = z$ which are already holding in $X$. $G(X)$ is constructed in the same way.

\begin{rema}
It should be noted that it will often happen that the map $\iota_G$ is not injective. It will not even suffice to assume that $S(X)$ is cancellative: the first example of a cancellative monoid which is not embeddable into any group has been constructed by Malcev \cite[§2]{malcev}.
\end{rema}

Using the universality of $\iota_S$ we immediately see that there is a unique homomorphism of monoids $i:S(X) \to G(X)$ with $\iota_G = i \circ \iota_S$.

\begin{defi}
If $X$ is an OML, we define its \emph{structure monoid} $S(X)$ as the structure monoid of the partial monoid $(X,\oplus)$, where $\oplus$ is defined as in \autoref{prop:omls_are_partial_monoids}.

Similarly, we define its \emph{structure group} as the structure group of the partial monoid $(X,\oplus)$.
\end{defi}

Rump discovered that structure groups of OMLs can be right lattice-ordered in a natural way:

\begin{thm} \label{thm:structure_groups_are_l_groups}
\cite[Corollary 4.6]{rump_goml} Let $X$ be an OML. Then $i(S(X))$ is the positive cone of a right lattice-order on $G(X)$. Furthermore, $\iota_G:X \to G(X)$ embeds $X$ as an interval of $G(X)$ under this partial order.
\end{thm}

\begin{rema}
The reader familiar with Rump's original article \cite{rump_goml} may have noted slight differences in this exposition. We explain these:
\begin{enumerate}[a)]
\item Rump originally defined the structure monoid $S(X)$ by the partial monoid $(X,\sqcap)$ which, however, is isomorphic to $(X,\oplus)$ via the isomorphism $x \mapsto x^{\ast}$. Furthermore, $S(X)$ is embedded as the \emph{negative} cone into $G(X)$ which does not influence the lattice-orderability of the subject. This is justified by the connection of his article to quantum logic: in this context, tautological truth, amounting to the element $1 \in X$, \emph{should} be identified with $e \in G$.

In this article we chose to \glqq turn things around\grqq\ because we found that the definitions are easier accessible when we define $S(X)$ as the positive cone whose multiplication is defined by adding orthogonal elements in $X$.
\item Our definitions of $S(X)$ and $G(X)$ by universal properties are actually a corollary in Rump's article. For we are interested in \emph{special} structure monoids resp. groups, we also decided in favour of the easier characterizations.
\end{enumerate}
\end{rema}

Rump also proved a converse to this theorem. To explain this result, we first need to introduce some special elements of a right $\ell$-group.

\begin{defi}
Let $G$ be a right $\ell$-group. Let $\Delta > e$. We call $\Delta$
\begin{enumerate}[a)]
\item \emph{normal} if $\Delta(g \vee h) = \Delta g \vee \Delta h$ holds for any $g,h \in G$ (that is, \emph{left}-multiplication by $\Delta$ is isotone), 
\item \emph{singular} if for $e \leq x,y \leq \Delta$, we always have the implication
\[
xy \leq \Delta \Rightarrow yx = x \vee y,
\]
\item a \emph{singular strong order unit} if $\Delta$ is normal and singular and, furthermore, for any $g \in G$ there exists $k \in \mathbb{Z}$ with $g \leq \Delta^k$.
\end{enumerate}
\end{defi}

\begin{rema}
\begin{enumerate}[1)]
\item In the classical theory of $\ell$-groups (without the predicate \emph{left} or \emph{right}), the concept of strong order unit already exists, see \cite[Definition 7.4.]{darnel}. There is also the possibility of defining \emph{weak} order units as elements $g\in G^+$ such that $g \wedge h \neq e$ for all $h \in G^+$ (see \cite[54.3.]{darnel}) but it is not known to the author if this concept is useful in the one-sided theory.
\item We \emph{might} define singularity by the implication $xy \leq \Delta \Rightarrow xy = x \vee y$. It turns out that, using this definition, the theory of right $\ell$-groups with singular strong order unit goes in a slightly different direction - see \cite[Section 5]{DRZ}.

We use this opportunity to note that singularity is also present in the classical theory, see \cite[Definition 6.9.]{darnel}, for example.
\end{enumerate}
\end{rema}

Now we can state the promised converse to \autoref{thm:structure_groups_are_l_groups}:

\begin{thm} \label{thm:which_groups_are_structure_groups}
\cite[Theorem 4.10]{rump_goml} For any OML $X$, the image $\iota_G(1)$ is a singular strong order unit in the right $\ell$-group $G(X)$.

Conversely, if $G$ is a right $\ell$-group with singular strong order unit $\Delta$, the interval $\left[ e, \Delta \right]$ becomes an OML under the lattice operations inherited by $G$ and the orthocomplementation defined by $g^{\ast} = g^{-1}\Delta$.

The embedding $\left[ e, \Delta \right] \hookrightarrow$ identifies $G$ as a structure group for the OML $\left[ e, \Delta \right]$.
\end{thm}

\section{The paraunitary group of an involutive Banach algebra} \label{sec:paraunitary_groups}

We give an overview on the basic definitions regarding $C^{\ast}$-algebras. Details can be found, for example, in \cite{dixmier_c_star}.

Let $\K \in \{ \R , \C \}$.

A \emph{Banach algebra} is a Banach space $(\A, \Vert . \Vert)$ (complex or real) with an associative, continuous bilinear map $\cdot: \A \times \A \to \A$; $(x,y)\mapsto xy$ which is submultiplicative, i.e. $\Vert xy \Vert \leq \Vert x \Vert \cdot \Vert y \Vert$. If there is a neutral element $1$ for this multiplication, we call it \emph{unital}.

An \emph{involution} on a Banach algebra is a map $\ast: \A \to \A ; x \mapsto x^{\ast}$ which is
\begin{enumerate}[i)]
\item an involution in the classical sense: $(x^{\ast})^{\ast} = x$ for all $x \in A$,
\item antilinear: for all $x,y \in \A$, $\alpha \in \K$,
\begin{align*}
(x + y)^{\ast} & = x^{\ast} + y^{\ast} \\
(\alpha x)^{\ast} & = \overline{\alpha} x^{\ast}
\end{align*}
(where the latter reduces to mere linearity when $\K = \R$),
\item an antihomomorphism: $(xy)^{\ast} = y^{\ast}x^{\ast}$ for all $x,y \in \A$,
\item norm-preserving: $\Vert x^{\ast} \Vert = \Vert x \Vert$.
\end{enumerate}

A Banach algebra with an involution is called a \emph{Banach-$\ast$-algebra}.

If additionally, we have $\Vert x x^{\ast} \Vert = \Vert x^{\ast}x \Vert = \Vert x \Vert^2$ for all $x \in \A$, we call $\A$ an $C^{\ast}$-algebra.

Note that every von Neumann algebra is a $C^{\ast}$-algebra under the operator norm and involution given by taking adjoints.

Throughout this section, $\A$ will denote an arbitrary Banach $\ast$-algebra (over $\R$ or $\C$)

We denote by $\A \laur$ the set of formal expressions
\[
\varphi = \sum_{i=-\infty}^{\infty}t^i \varphi_i
\]
with $\varphi_i \in \A$ and $\varphi_i \neq 0$ for only finitely many $i$. Furthermore, $\A \polyp$ (resp. $\A \polym$) will denote the subsets consisting of all $\varphi \in \A \laur$ with $\varphi_i = 0$ for $i < 0$ ($i > 0$).

There is a natural way of defining a multiplication on $\A \laur$, extending that of $\A$: for $\varphi, \psi \in \A \laur$, we set
\[
\varphi \psi = \sum_{i = -\infty}^{\infty} t^i \left( \sum_{k = -\infty}^{\infty} \varphi_k \psi_{i-k} \right).
\]
This is just the ring of finite Laurent series over $\A$ which clearly is associative.

The involution $\ast$ can be extended to elements of $\A\laur$ by the rule
\[
\varphi^{\ast} = \sum_{i = -\infty}^{\infty}t^{-i}a_i^{\ast}.
\]
Clearly, $\ast$ is an involution and an anti-homomorphism of $\A\laur$, therefore making the latter an involutive $\K$-algebra\footnote{Which, however, is \emph{not} a Banach $\ast$-algebra, due to a lack of completeness}. Restricting $\ast$ from $\A\laur$ to the subalgebra $\A$, we get the familiar involution on the latter, therefore there will be no danger of confusion when using the same symbol \glqq $\ast$\grqq\ for these operators.

Note that $\A\polyp^{\ast} = \A\polym$ and $\A\polym^{\ast} = \A\polyp$.

There is a canonical specialization map
\begin{align*}
\varepsilon_1: \A \laur & \to \A \\
\varphi & \mapsto \sum_{i = -\infty}^{\infty} \varphi_i 
\end{align*}
which amounts to setting $t = 1$. The reader can easily convince himself that $\varepsilon_1$ respects the involutive algebra structures in the sense that it is a homomorphism of algebras which additionally fulfils $\varepsilon_1(\varphi^{\ast}) = \varepsilon_1(\varphi)^{\ast}$.

We can now come to our main definitions:

With $\A$ we can associate its \emph{unitary group} which we define as
\[
U(\A) := \left\{ x \in \A: x^{\ast}x = xx^{\ast} = 1 \right\}.
\]
This is easily seen to become a group under the multiplication of $\A$.

Similarly, we can associate quite another group with $\A$, namely, its \emph{paraunitary group}
\begin{equation} \label{eq:def_of_pua}
\PU(\A) := \left\{ \varphi \in \A\laur : \varphi^{\ast} \varphi = \varphi \varphi^{\ast} = 1 \right\}.
\end{equation}

Let $\varphi \in \PU(\A)$. We then have
\[
\varepsilon_1(\varphi)^{\ast}\varepsilon_1(\varphi) = \varepsilon_1(\varphi^{\ast})\varepsilon_1(\varphi) = \varepsilon_1(\varphi^{\ast}\varphi) = \varepsilon_1(1) = 1.
\]
Similarly, one shows $\varepsilon_1(\varphi)\varepsilon_1(\varphi)^{\ast} = 1$, thus proving that $\varepsilon_1(\varphi) \in U(\A)$.

On the other hand, we have a canonical embedding of rings
\begin{align*}
\iota: \A & \to \A\laur \\
x & \mapsto t^0 x
\end{align*}
which is a morphism of rings and clearly commutes with the respective $\ast$-operations.

From this one easily sees that $\iota(x) \in \PU(\A)$ whenever $x \in U(\A)$ and that these elements comprise the totality of constant paraunitary polynomials.

Furthermore, $\varepsilon_1 \iota = 1_{\A}$ which proves

\begin{prop} \label{prop:split_exact_sequence_PUA}
There is a split exact sequence of groups
\[
1 \longrightarrow \ker \varepsilon_1 \longrightarrow \PU(\A) \overset{\varepsilon_1}{\longrightarrow} U(\A) \longrightarrow 1.
\]
with $\iota$ as a right inverse to $\varepsilon_1$.
\end{prop}

We can now introduce the main character of this section:

\begin{defi}
The \emph{pure paraunitary group} associated with $\A$ is the group $\PPU(\A):= \ker \varepsilon_1$.
\end{defi}

By a classical group theoretic argument one deduces from \autoref{prop:split_exact_sequence_PUA} the
\begin{cor}
There is a semidirect product decomposition $\PU(\A) = \PPU(\A) \rtimes U(\A)$ where we identify $U(\A)$ with the constant paraunitary polynomials.
\end{cor} 
\begin{rema}
\begin{enumerate}[1)]
\item The designation of the elements of $\ker \varepsilon_1$ as \emph{pure} is due to the following analogy with braid groups (\cite[Section 1.3]{Kassel_braid}):

Denoting the braid group with $n$ strands by $B_n$ and the symmetric group on $n$ elements by $S_n$ there is a canonical specialization map $\varepsilon: B_n \to S_n$ fitting into a short exact sequence
\[
1 \to P_n \rightarrow B_n \overset{\varepsilon}{\rightarrow} S_n
\]
where $P_n := \ker(\varepsilon)$ is the so-called \emph{pure} braid group.
\item $\varepsilon_1$, however, is not the only possibility of defining a pure paraunitary group: for any $z \in \C$ with $|z|=1$ there is a specialization map
\begin{align*}
\varepsilon_z: \A\laur & \to \A \\
\varphi & \mapsto \sum_{i=-\infty}^{\infty}z^i \varphi_i
\end{align*}
which amounts to setting $t = z$. One can show that each $\varepsilon_z$ restricts to a homomorphism $\PU(\A) \to U(\A)$.

If we denote the respective kernels by $\PPU_z(\A)$ one can show that each $\PPU_z(\A)$ are isomorphic to $\PPU(\A)$ by the isomorphism
\begin{align*}
\alpha_z: \PPU(\A) & \to \PPU_z(\A) \\
\sum_{i=-\infty}^{\infty}t^i \varphi_i & \mapsto \sum_{i=-\infty}^{\infty}t^i z^{-i} \varphi_i.
\end{align*}
Therefore, if one knows one member of the family $\PPU_z(\A)$ one knows them all.

It may, however, lead to interesting questions when one applies $\varepsilon_{z^{\prime}}$ to $\PPU_z(\A)$ when $z, z^{\prime}$ are not necessarily equal. For an instance of the case $z=1,z^{\prime}=-1$, see \cite[Section 5.1]{dietzel}.
\end{enumerate}
\end{rema}

The groups $\PPU(\A)$ are ordered in a very natural way:
We can define the subsets $\PPU^+(\A) := \PPU(\A) \cap \A\polyp$ and $\PPU^-(\A) := \PPU(\A) \cap \A\polym$ which have the following pleasant property:

\begin{prop} \label{prop:ppu_plusminus_are_cones}
The subsets $\PPU^+(\A), \PPU^-(\A)$ are the positive respectively negative cone of a right-invariant order on $\PPU(\A)$.
\end{prop}

\begin{proof}
We first prove that $\PPU^+(\A)$ is a positive cone. By \autoref{prop:pure_submonoids_are_cones}, we have to show that it is a pure submonoid:
First of all, $\PPU^+(\A)$ is a submonoid of $\PPU(\A)$ which comes from the fact that both $\PPU(\A)$ and $\A\polyp$ contain $1$ and are multiplicatively closed in $\A\laur$.

On the other hand, for $\varphi \in \PPU^+(\A)$ we have, by (\ref{eq:def_of_pua}), that $\varphi^{-1} = \varphi^{\ast} \in \PPU^-(\A)$.

Therefore $\PPU^+(\A) \cap \left( \PPU^+(\A) \right)^{-1} = \PPU^+(\A) \cap \PPU^-(\A)$, the latter consisting only of constant paraunitary polynomials, i.e. those of the form $t^0 x$ ($x \in \A$). But these have to fulfil $1 = \varepsilon_1(t^0 x) = x$. Thus $\PPU^+(\A) \cap \left( \PPU^+(\A) \right)^{-1} = 1$.

Finally, from $\left( \PPU^+(\A) \right)^{-1} = \PPU^-(\A)$ one sees that $\PPU^-(\A)$ is the negative cone of the right-invariant order defined by the positive cone $\PPU^+(\A)$.
\end{proof}

Before diving into our analysis of the pure paraunitary groups for general von Neumann algebras we give two motivating examples:

\begin{exam}
\begin{enumerate}[1)]
\item By means of the Gelfand isomorphism, any commutative, unital $C^{\ast}$-algebra $\A$ is isomorphic (in the sense of $C^{\ast}$-algebras) to the $C^{\ast}$-algebra $C(X)$ of complex-valued continuous functions on some compact Hausdorff space $X$, the $\ast$-operation being complex conjugation (\cite[Theorem 1.4.1.]{dixmier_c_star}).

By this isomorphism, each $\varphi \in \A\laur$ can be seen as a family of finite Laurent polynomials $\varphi_x(t) \in \C\laur$ ($x \in X$) such that the coefficients of $\varphi_x(t)$ vary continuously in $x$ and there is a global bound for the degrees of the least and the greatest non-zero coefficients of each $\varphi_x$.

The paraunitarity condition now reads as $\varphi_x(t)\cdot \bar{\varphi}_x(t^{-1}) = 1$ which is only possible when $\varphi_x(t) = z_xt^{n_x}$ with $\vert z_x\vert = 1$ and $n_x \in \mathbb{Z}$. Thus, the coefficients of $\varphi_x(t)$ are either $0$ or have absolute value $1$ from which we infer that $n_x$ must be locally constant, which is equivalent to saying that $n_x$ is continuous when $\mathbb{Z}$ carries the discrete topology.

Setting $t = 1$ in $\varphi_x(t) = z_x t^{n_x}$, one sees that $\varphi_x(t)$ represents an element of $\PPU(\A)$ if and only if $z_x = 1$ for all $x$. In this case, $\varphi_x(t) = t^{n_x}$.

Let $C(X,\mathbb{Z})$ be the (additive) group of continuous functions $n:X \to \mathbb{Z}$. Sending such a function $n$ to the family $\varphi_x(t)=t^{n_x}$ establishes an isomorphism of abstract groups $C(X,\mathbb{Z}) \cong \PPU(\A)$.

The elements of $\PPU^+(\A)$ are then represented by families $\varphi_x(t) = t^{n_x}$ with $n_x \in \mathbb{Z}_0^+$ varying continuously. This shows that the isomorphism in the above paragraph identifies the submonoid $C(X,\mathbb{Z}_0^+)$ of $C(X,\mathbb{Z})$ with $\PPU^+(\A)$.

We conclude that, as a right-ordered group, $\PPU(\A)$ is isomorphic to the additive group $C(X,\mathbb{Z})$ under the pointwise partial order.
\item If $\A$ is a commutative von Neumann algebra, $\A$ is $\ast$-isomorphic to the algebra of essentially bounded, complex-valued, measurable functions on some measure space $(Y,\mu)$ \cite[I.7.3.,Theorem 1]{Dixmier_Neumann}.

By a similar argument as the one given above, we can identify the right-ordered group $\PPU(\A)$ with $L^{\infty}(Y,\mathbb{Z})$, that is, the functions in $L^{\infty}(Y)$ which are integral $\mu$-almost everywhere, and where we define $f \leq g$ if and only if $f(x) \leq g(x)$ $\mu$-almost everywhere.
\end{enumerate}
\end{exam}

\section{The pure paraunitary group of a von Neumann algebra} \label{sec:ppu_is_lattice_ordered}

In this section, $\A$ will always be assumed to be a von Neumann algebra acting on a complex Hilbert space $\Hil$. Both will be fixed throughout this section.

The aim of this section is to give a proof of the following

\begin{thm} \label{thm:PPU_is_lattice_ordered}
If $\A$ is a von Neumann algebra, the right-invariant order on $\PPU(\A)$ defined by the positive cone $\PPU^+(\A)$ makes the former a lattice.
\end{thm}

\subsection{Preliminary constructions}

We begin by constructing several Hilbert spaces for the rings $\A^{\prime}\laur, \A^{\prime} \polyp, \A^{\prime}\polym$ to act on.

First of all, we set
\[
\Hill := \left\{ x = \sum_{i = -\infty}^{\infty} t^i x_i : x_i \in \Hil, \sum_{i = -\infty}^{\infty} \Vert x_i \Vert^2 < \infty \right\}
\]

$\Hill$ becomes a Hilbert space as follows:

For $x = \sum_{i = -\infty}^{\infty} t^i x_i$ and $y = \sum_{i = -\infty}^{\infty} t^i y_i$ in $\Hill$ we define their product as
\begin{equation} \label{eq:hill_inner_product}
\left\langle x,y \right\rangle_{\Hill} = \sum_{i = -\infty}^{\infty} \left\langle x_i, y_i \right\rangle_{\Hil}.
\end{equation}
We will suppress the subscript $\Hill$ of the inner product whenever it is clear if we are concerned with elements of $\Hil$ or $\Hill$. Clearly, $\Hill$ together with its inner product is immediately seen to be equivalent to the infinite direct sum $\oplus_{i=-\infty}^{\infty} \Hil$. It is well-known that the result of this construction is also a Hilbert space.

We define the closed subspaces $\Hilp, \Hilm \subseteq \Hill$ as the collection of all $x \in \Hill$ fulfilling $x_i = 0$ for $i < 0$ (resp. $i > 0$).

$\Hill$ becomes an $\A^{\prime} \laur$-module in a natural way: for $\varphi \in \A^{\prime} \laur, x \in \Hill$, we simply set
\[
\varphi x = \sum_{i = -\infty}^{\infty} t^i \left( \sum_{k = -\infty}^{\infty} \varphi_k x_{i-k} \right)
\]

When defining a module action on a Hilbert space, it should be proved that the multiplication maps are continuous, i.e.

\begin{prop}
For each $\varphi \in \A\laur$ or $\varphi \in \A^{\prime}\laur$, the map $x \mapsto \varphi x$ is continuous on $\Hil \powel$.
\end{prop}

\begin{proof}
It clearly suffices to assume $\varphi \in \A\laur$.

$x \mapsto tx$ and $x \mapsto t^{-1}x$ are just shift operators on $\Hill$ which are clearly isometries and therefore continuous.

For $\varphi_0 \in \A$, we have
\[
\Vert t^0 \varphi x \Vert = \Vert \sum_{i=-\infty}^{\infty} t^i \varphi_0 x_i \Vert = \left( \sum_{i=-\infty}^{\infty} \Vert \varphi_0 x_i \Vert^2 \right)^{\frac{1}{2}} \leq \Vert \varphi_0 \Vert \left( \sum_{i=-\infty}^{\infty} \Vert x_i \Vert^2 \right)^{\frac{1}{2}} = \Vert \varphi_0 \Vert \cdot \Vert x \Vert,
\]
so $t^0 \varphi_0$ acts continuously, too. The action of an arbitrary element of $\A \laur$ is a finite combination of these operations and therefore continuous.
\end{proof}

The inner product is connected with the $\ast$-operation on $\A\laur$ by the equation
\begin{equation} \label{eq:adjointness_of_star}
\left\langle x, \varphi y \right\rangle = \left\langle \varphi^{\ast} x, y \right\rangle
\end{equation}
which can be reduced to the case $\varphi = t^k \varphi_k$, using linearity. This case follows from the calculation
\begin{align*}
\left\langle x, t^k \varphi_k y \right\rangle & = \left\langle \sum_{i = -\infty}^{\infty} t^i x_i, \sum_{i = -\infty}^{\infty} t^i \varphi_k y_{i-k} \right\rangle \\
& = \sum_{i = -\infty}^{\infty} \left\langle x_i, \varphi_k y_{i-k} \right\rangle \\
& = \sum_{j = -\infty}^{\infty} \left\langle \varphi_k^{\ast} x_{j+k}, y_j \right\rangle \\
& = \left\langle \sum_{j = -\infty}^{\infty} t^j \varphi_k^{\ast} x_{j+k}, \sum_{j = -\infty}^{\infty} t^j y_j \right\rangle \\
& = \left\langle t^{-k}\varphi_k^{\ast} x, y \right\rangle.
\end{align*}

It turns out, that $\Hilp$ (resp. $\Hilm$) is invariant under the action of $\A^{\prime} \polyp$ (resp. $\A^{\prime} \polym$).

How does $\PPU(\A)$ come into play? This question is partly answered by the following proposition:

\begin{prop} \label{prop:ppu_acts_unitarily}
Each $\varphi \in \PPU(\A)$ acts as a unitary operator - with respect to the inner product given by (\ref{eq:hill_inner_product}) - on $\Hill$.
\end{prop}

\begin{proof}
Let $\varphi \in \PPU(\A)$, $x,y \in \Hill$, then $\left\langle \varphi x , \varphi y \right\rangle \overset{\text{(\ref{eq:adjointness_of_star})}}{=} \left\langle \varphi^{\ast} \varphi x , y \right\rangle = \left\langle x, y \right\rangle$.
\end{proof}

\begin{rema}
A reader who is interested in the analytical aspects may feel uneasy with the rather \glqq algebraic\grqq\ definition of $\PPU(\A)$ in terms of finite Laurent series. (Pure) paraunitary groups with a rather \glqq analytic\grqq\ flavour will be discussed in \autoref{sec:remarks}.
\end{rema}

We can now introduce a family of lattices which will play a crucial role in this section:

\begin{defi}
For integers $m \leq n$ we define $\sub(\Hil)_m^n$ as the lattice of all $\A^{\prime}\polym$-invariant closed subspaces $M \subseteq \Hill$ such that $t^m \Hilm \subseteq M \subseteq t^n \Hilm$.

Furthermore, for $m \in \Z$ we define
\begin{align*}
\sub(\Hil)_m^{\infty} & = \bigcup_{n = m}^{\infty}  \sub(\Hil)_m^n, \\
\sub(\Hil)_{-\infty}^m & = \bigcup_{n = -\infty}^m  \sub(\Hil)_m^n, \\
\sub(\Hil)_{-\infty}^{\infty} & = \bigcup_{n = -\infty}^{\infty}  \sub(\Hil)_{-n}^n
\end{align*}
where we see each $\sub(\Hil)_m^n$ as a sublattice in the lattice of all closed subspaces of $\Hill$.
\end{defi}

From now on, when speaking of $\PPU(\A)$ or $\PPU^+(\A)$ as ordered sets, we will always mean the right-invariant orders defined by $\varphi \leq \varphi^{\prime}$ whenever there is a $\psi \in \PPU^+(\A)$ such that $\varphi^{\prime} = \psi \varphi$.

Another protagonist in this section is the mapping $\Omega$ which is defined as
\begin{align*}
\Omega: \PPU(\A) & \to \sub(\Hil)_{-\infty}^{\infty} \\
\varphi & \mapsto \varphi \Hilm
\end{align*}
which has a sibling $\Omega^+$ which is similarly defined as
\begin{align*}
\Omega^+: \PPU^+(\A) & \to \sub(\Hil)_0^{\infty} \\
\varphi & \mapsto \varphi \Hilm.
\end{align*}

First of all, we have to show that $\Omega$ and $\Omega^+$ are indeed mappings, that is:

\begin{prop} \label{prop:omegas_are_welldefined_and_monotone}
$\Omega$ and $\Omega^+$ are well defined and monotone.
\end{prop}

\begin{proof}
We start by showing that $\Hilm \subseteq \varphi \Hilm$ for $\varphi \in \PPU^+(\A)$. This is equivalent to $\varphi^{-1} \Hilm \subseteq \Hilm$ which follows from $\varphi^{-1} = \varphi^{\ast} \in \A \polym$.

Take $\varphi, \varphi^{\prime} \in \PPU^+(\A)$ with $\varphi \leq \varphi^{\prime}$. There is a $\psi \in \PPU^+(\A)$ with $\varphi^{\prime} = \varphi \psi$. From what we have shown before it follows that
\[
\varphi^{\prime} \Hilm = \varphi \psi \Hilm \supseteq \varphi \Hilm.
\]

We can now show that $\varphi \in \PPU^+(\A)$ implies $\varphi \Hilm \subseteq t^n \Hilm$ for some integer $n$: take a non-negative $n$ with $t^n \varphi^{\ast} \in \PPU^+(\A)$. Then $\varphi t^n \varphi^{\ast} = t^n$, therefore $\varphi \leq t^n \Rightarrow \varphi \Hilm \subseteq t^n \Hilm$.

So, $\Omega^+$ is well-defined and monotone.

For general $\varphi \in \PPU(\A)$ chose an integer $m$ such that $t^m \varphi \in \PPU^+(\A)$. From the arguments above follows $\Hilm \subseteq t^m \varphi \Hilm \subseteq t^n \Hilm$ for some non-negative integer $n$. Therefore, $t^{-m} \Hilm \subseteq \varphi \Hilm \subseteq t^{n-m} \Hilm$, proving that $\Omega$ is well-defined.

Assuming $\varphi, \varphi^{\prime} \in \PPU(\A)$ (but still $\psi \in \PPU^+(\A)$) in the proof of monotonicity for $\Omega^+$, we get a proof for the monotonicity of $\Omega$.
\end{proof}

The aim of the following two subsections will be the proof of

\begin{thm} \label{thm:omega_is_bijective}
$\Omega$ and $\Omega^+$ are bijective.
\end{thm}

\begin{proof}
In the following subsections we will show that $\Omega^+$ is both injective (\autoref{prop:omega_is_injective}) and surjective (\autoref{prop:omega_is_surjective}). This proves the second part.

The bijectivity of $\Omega$ is then an easy corollary:

Let $\varphi, \varphi^{\prime} \in \PPU(\A)$ fulfil $\varphi \Hilm = \varphi^{\prime} \Hilm$. Chose an integer $n$ with $t^n \varphi, t^n \varphi^{\prime} \in \PPU^+(\A)$. Clearly, we also have $t^n \varphi \Hilm = t^n \varphi^{\prime} \Hilm$. From the injectivity of $\Omega^+$ it follows that $t^n \varphi = t^n \varphi^{\prime}$ and finally $\varphi = \varphi^{\prime}$. This proves that $\Omega$ is injective.

Take $M \in \sub(\Hil)_{-\infty}^{\infty}$. Then there is an integer $n$ with $t^n M \in  \sub(\Hil)_{0}^{\infty}$. $\Omega^+$ is surjective, therefore there is some $\varphi \in \PPU^+(\A)$ with $\varphi \Hilm = t^n M$, showing that $t^{-n}\varphi \Hilm = M$.
\end{proof}

\begin{rema}
$\Omega$ and $\Omega^+$ may be bijective and order-preserving but it does not already follow that they are order-isomorphisms - this will be shown later (\autoref{lem:omega_is_an_embedding}).
\end{rema}

\subsection{\texorpdfstring{$\Omega^+$}{Omega+} is injective}

We begin with a simple lemma:

\begin{lem} \label{lem:sending_hilm_to_hilm}
Let $\psi \in \PPU(\A)$ fulfil $\psi \Hilm \subseteq \Hilm$. Then $\psi \in \PPU^-(\A)$:
\end{lem}

\begin{proof}
Let $\psi \Hilm \subseteq \Hilm$ and assume $\psi \notin \PPU^-(\A)$. Then there is a $k>0$ such that $\psi = \sum_{i=-\infty}^k t^i \psi_i$ with $\psi_k \neq 0$.

Take an $x \in \Hil$ with $\psi_k x \neq 0$. Then $t^0 x \in \Hilm$ but $\psi (t^0 x) = \sum_{i=-\infty}^k t^i \psi_i x$ which has the nonzero $k$'th coefficient $\psi_k x$. Therefore, $\psi(t^0x) \notin \Hilm$ which is a contradiction.
\end{proof}

It is now easy to show:

\begin{prop} \label{prop:omega_is_injective}
$\Omega^+$ is injective.
\end{prop}

\begin{proof}
Take $\varphi, \varphi^{\prime} \in \PPU^+(\A)$ with $\varphi \Hilm = \varphi^{\prime} \Hilm$. From this it follows that $\varphi^{-1} \varphi^{\prime} \Hilm = \Hilm$ and $\varphi^{\prime -1} \varphi \Hilm = \Hilm$.

\autoref{lem:sending_hilm_to_hilm} implies that $\varphi^{-1} \varphi^{\prime} \in \PPU^-(\A)$ and $(\varphi^{-1} \varphi^{\prime})^{-1} = \varphi^{\prime-1} \varphi \in \PPU^-(\A)$. But $\PPU^-(\A)$ is a negative cone (\autoref{prop:ppu_plusminus_are_cones}), therefore $\varphi^{-1} \varphi^{\prime} = 1$ which implies $\varphi = \varphi^{\prime}$.
\end{proof}

\subsection{\texorpdfstring{$\Omega^+$}{Omega+} is surjective}

We can define, for any $M \in X(\A^{\prime})$, an element $p_M := t\pi_M + \pi_{M^{\ast}} \in \A\polyp$ where the latter membership is guaranteed by \autoref{prop:invariant_subspaces_are_annihilated}.

We can even say more about these elements:

\begin{prop}
For each $M \in X(\A^{\prime})$, we have $p_M \in \PPU^+(\A)$.
\end{prop}

\begin{proof}
Using the self-adjointness of projections, we calculate:
\begin{align*}
p_M^{\ast} p_M  & = \left( t^{-1}\pi_M + \pi_{M^{\ast}} \right) \left( t\pi_M + \pi_{M^{\ast}} \right) \\
& = t^{-1} \underbrace{\pi_M \pi_{M^{\ast}}}_{=0} +\pi_{M}^2 + \pi_{M^{\ast}}^2 +  t \underbrace{ \pi_{M^{\ast}}\pi_M}_{=0} \\
& = \pi_M + \pi_{M^{\ast}} = 1.
\end{align*}
Similarly, one calculates $p_M p_M^{\ast} = 1$, thus showing that $p_M$ is indeed paraunitary.

Finally, $\varepsilon_1(p_M) = \pi_M + \pi_{M^{\ast}} = 1$.
\end{proof}

For $M \in X(\A^{\prime})$, we define for an integer $i$ the closed subspace
\[
t^i M := \left\{ t^i x: x \in M \right\} \subseteq \Hill.
\]

It turns out that we can already specify what $\Omega(p_M)$ is:

\begin{lem} \label{lem:what_is_omega_pm}
For $M \in X(\A^{\prime})$, $\Omega(p_M) = \Hilm \oplus tM$.
\end{lem}

\begin{proof}
$tM$ is contained in $t \Hilp = \Hilm^{\ast}$ therefore the sum is indeed orthogonal.

We must show $p_M \Hilm = \Hilm \oplus tM$.

Take $x = \sum_{i=-\infty}^0 t^i x_i \in \Hilm$, then an easy calculation shows that
\[
p_M x = t \pi_M x_0 + \sum_{i=-\infty}^0 t^i \left( \pi_M x_{i-1} + \pi_{M^{\ast}}x_i \right) \in \Hilm \oplus tM,
\]
thus proving $p_M \Hilm \subseteq \Hilm \oplus tM$.

In order to prove $p_M \Hilm \supseteq \Hilm \oplus tM$ we must show $p_M^{-1} \left( \Hilm \oplus tM \right) \subseteq \Hilm$.

Let $x = tx_1 + \sum_{i=-\infty}^0 t^i x_i \in \Hilm \oplus tM $, that is, $x_1 \in M$. Furthermore $p_M^{-1} = p_M^{\ast} = t^{-1} \pi_M + \pi_{M^{\ast}}$. We calculate:
\[
p_M^{-1} x = t \underbrace{\pi_{M^{\ast}}x_1}_{=0} + \sum_{i = -\infty}^0 t^i \left( \pi_M x_{i+1} + \pi_{M^{\ast}} x_i \right) \in \Hilm
\]
which proves the other inclusion.
\end{proof}

\begin{prop} \label{prop:omega_is_surjective}
$\Omega^+$ is surjective.
\end{prop}

\begin{proof}
Take $M \in \sub(\Hil)_0^{\infty}$. We will show that there is a $\varphi \in \PPU^+(\A)$ with $\varphi^{-1} M = \Hilm $. That clearly proves our proposition.

There is an $n \geq 0$ with $M \in \sub(\Hil)_0^n$. The proof will be by induction over this $n$.

If $n = 0$ then $M = \Hilm = 1\cdot \Hilm$.

Assume that for fixed $N \geq 0$ we have shown that for each $M \in \sub(\Hil)_0^N$ there is a $\varphi \in \PPU^+(\A)$ with $\varphi^{-1}M = \Hilm$.

Now take $M \in \sub(\Hil)_0^{N+1}$. $M$ is invariant under multiplication by $t^{-1}$ and contains $\Hilm$, so:
\begin{equation} \label{eq:eq_in_proof_of_surjectivity}
t^{-N} M + \Hilm \subseteq M \cap t \Hilm.
\end{equation}

There is a closed $\A^{\prime}$-invariant subspace $M_1$ such that $M \cap t \Hilm = \Hilm \oplus tM_1$ - this follows from the following consideration:

The orthomodular law for Hilbert spaces, together with $\Hilm \subseteq M \cap t \Hilm$, implies that
\[
M \cap t \Hilm = \Hilm \oplus \left(\Hilm^{\ast} \cap t \Hilm \cap M \right).
\]
The second summand can be determined further:
\[
\Hilm^{\ast} \cap t \Hilm \cap M = t \Hilp \cap \Hilm \cap M = t \Hil \cap M
\]
which is of the form $tM_1$ for some $M_1 \in X(\A^{\prime})$.

Therefore $\Hilm \oplus tM_1 \subseteq M$. From \autoref{lem:what_is_omega_pm}, we deduce
\[\Hilm = p_{M_1}^{-1} \left( \Hilm \oplus tM_1 \right) \subseteq p_{M_1}M.
\]

We claim that $p_{M_1}^{-1} M \in \sub(\Hil)_0^N$, so it remains to show that also $p_{M_1}^{-1} M \subseteq t^N \Hilm$. Using \autoref{eq:eq_in_proof_of_surjectivity}, we calculate:
\begin{align*}
t^{-N} p_{M_1}^{-1} M & = p_{M_1}^{-1} t^{-N} M \\
& \subseteq p_{M_1}^{-1} \left( t^{-N} M + \Hilm \right) \\
& \overset{\text{(\ref{eq:eq_in_proof_of_surjectivity}})}{\subseteq} p_{M_1}^{-1} \left( M \cap t \Hilm \right) \\
& = p_{M_1}^{-1} \left( \Hilm \oplus tM_1 \right) = \Hilm.
\end{align*}

By our induction hypothesis, there is a $\varphi \in \PPU^+(\A)$ with $\varphi^{-1} p_{M_1}^{-1} M = \Hilm$, i.e. $(p_{M_1}\varphi)^{-1}M = \Hilm$.
\end{proof}

\subsection{Proof of \autoref{thm:PPU_is_lattice_ordered}}

We are finished with our proof of \autoref{thm:PPU_is_lattice_ordered} as soon as we can pull back the lattice structure of $\sub(\Hil)_{-\infty}^{\infty}$ to $\PPU(\A)$. We need the following
\begin{lem} \label{lem:omega_is_an_embedding}
Let $\varphi, \varphi^{\prime} \in \PPU(\A)$ with $\Omega(\varphi) \subseteq \Omega(\varphi^{\prime})$. Then $\varphi \leq \varphi^{\prime}$.
\end{lem}

\begin{proof}
Our condition means that $\varphi \Hilm \subseteq \varphi^{\prime} \Hilm$ which is equivalent to $\Hilm \subseteq \varphi^{-1} \varphi^{\prime} \Hilm$.

By \autoref{prop:omega_is_surjective}, there is a $\psi \in \PPU^+(\A)$ with $\psi \Hilm = \varphi^{-1} \varphi^{\prime} \Hilm$. \autoref{prop:omega_is_injective} now tells us that $\psi = \varphi^{-1} \varphi^{\prime}$ resp. $\varphi \psi = \varphi^{\prime}$ which is exactly what we wanted to show.
\end{proof}

\begin{proof}[Proof of \autoref{thm:PPU_is_lattice_ordered}]
\autoref{thm:omega_is_bijective}, \autoref{prop:omegas_are_welldefined_and_monotone} and \autoref{lem:omega_is_an_embedding} together say that $\Omega: \PPU(\A) \to \sub(\Hil)_{-\infty}^{\infty}$ is a bijective embedding of ordered sets, i.e. it is isotone. But $\sub(\Hil)_{-\infty}^{\infty}$ is a lattice, and so is $\PPU(\A)$.
\end{proof}

We will need the result later, therefore we cite the following part of the proof given above as a corollary:

\begin{cor} \label{cor:omega_is_a_lattice_isomorphism}
$\Omega: \PPU(\A) \to \sub(\Hil)_{-\infty}^{\infty}$ is an isomorphism of lattices.
\end{cor}

\section{The structure group of a von Neumann algebra} \label{sec:ppu_is_structure_group}

Here we will prove that the notions of structure group and pure paraunitary group of a von Neumann algebra actually coincide.

We need a lemma before. Recall that, for $M \in X(\A^{\prime})$, we defined $p_M = t\pi_M + \pi_{M^{\ast}}$.

\begin{lem} \label{lem:divisors_of_t}
Let $\varphi \in \PPU^+(\A)$. Then $\varphi \leq t$ if and only if $\varphi = p_M$ for some $M \in X(\A^{\prime})$.
\end{lem}

\begin{proof}
By \autoref{cor:omega_is_a_lattice_isomorphism}, $1 \leq \varphi \leq t$ is equivalent to $\Hilm \subseteq \varphi \Hilm \subseteq t \Hilm$.

By the orthomodular law,
\[
\varphi \Hilm = \Hilm \oplus (\Hilm^{\ast} \cap \varphi \Hilm) = \Hilm \oplus (t \Hilp \cap \varphi \Hilm)
\]
and $t \Hilp \cap \varphi \Hilm \subseteq t \Hilp \cap t \Hilm = t \Hil$. $t \Hilp \cap \varphi \Hilm$ is easily seen to be closed and invariant under $\A^{\prime} \subseteq \A^{\prime}\laur$ and therefore is of the form $tM$ where $M \subseteq \Hil$ is $\A^{\prime}$-invariant.

Therefore $\varphi \Hilm = \Hilm \oplus tM = p_M \Hilm$ (\autoref{lem:what_is_omega_pm}), and therefore $\varphi = p_M$ by \autoref{prop:omega_is_injective}.

On the other hand, for $M \in X(\A^{\prime})$, we have $tp_M^{\ast} \in \PPU^+(\A)$ and $tp_M^{\ast}p_M = t$, from which $p_M \leq t$ follows.
\end{proof}

We proceed with a lemma:

\begin{lem} \label{lem:multiplication_of_pms}
For $M,N \in X(\A^{\prime})$ with $M \bot N$ we have $p_M p_N = p_{M \oplus N}$.
\end{lem}

\begin{proof}
For $M \bot N$, it is known that $\pi_M \pi_N = 0$ and $\pi_M + \pi_N = \pi_{M \oplus N}$.

First of all, we deduce
\[
\pi_{M^{\ast}}\pi_{N^{\ast}} = (1-\pi_M)(1-\pi_N) = 1 - (\pi_M + \pi_N) + \pi_M \pi_N = 1 - \pi_{M \oplus N} = \pi_{(M \oplus N)^{\ast}}
\]
and
\[
\pi_M \pi_{N^{\ast}} = \pi_M (1-\pi_N) = \pi_M - \pi_M \pi_N = \pi_M.
\]
Similarly, $\pi_{M^{\ast}} \pi_N$.

Using these formulae, we calculate:
\begin{align*}
p_M p_N & = (t \pi_M + \pi_{M^{\ast}})(t \pi_N + \pi_{N^{\ast}}) \\
& = t^2 \pi_M \pi_N + t(\pi_M \pi_{N^{\ast}}+\pi_{M^{\ast}} \pi_N) + \pi_{M^{\ast}}\pi_{N^{\ast}} \\
& = t(\pi_M + \pi_N) + \pi_{(M \oplus N)^{\ast}} \\
& = t \pi_{M \oplus N} + \pi_{(M \oplus N)^{\ast}} = p_{M \oplus N}.
\end{align*}
\end{proof}

\begin{cor} \label{cor:gamma_is_an_isomorphism}
The map defined by
\begin{align*}
\Gamma: X(\A^{\prime}) & \to \left[ 1, t \right] \subseteq \PPU(\A) \\
M & \mapsto p_M
\end{align*}
is an isomorphism of lattices.
\end{cor}

\begin{proof}
We construct the inverse map:

By \autoref{lem:divisors_of_t}, $\left[1,t \right]$ consists exactly of the $p_M$ with $M \in X(\A^{\prime})$.

$\Omega^+$ gives an isomorphism between the intervals $\left[ 1,t \right]$ and $\left[ \Omega(1), \Omega(t) \right]$. But each $\A^{\prime}\polym$-invariant subspace between $\Omega(1) = \Hilm$ and $\Omega(t) = t \Hilm$ is of the form $\Hilm \oplus tM$ with $M \in X(\A^{\prime})$.

Therefore $\left[ \Omega(1), \Omega(t) \right] \cong X(\A^{\prime})$ as lattices.

It follows that mapping $p_M$ to $\Omega(p_M) = \Hilm \oplus tM$ (\autoref{lem:what_is_omega_pm}) and then to $M$ provides us with $\Gamma^{-1}$.
\end{proof}

\begin{thm}
$\PPU(\A)$ with the right-invariant lattice-order defined by the positive cone $\PPU^+(\A)$, is isomorphic - in the sense of right-ordered groups - to $G(X(\A^{\prime}))$.
\end{thm}

\begin{proof}
We show that $t$ is a singular strong order unit in $\PPU(\A)$:

Clearly, $t$ is positive. $t$ is also normal: right-multiplication by $t$ is clearly order-preserving, due to $\PPU(\A)$ being right-ordered, and so is left-multiplication because $t$ commutes with each element of $\A \laur$.

Let $\varphi \in \PPU(\A)$. We show that it right-divides some power of $t$: let $n$ be the greatest integer with $\varphi_n \neq 0$. Then $\psi := t^n \varphi^{\ast} \in \PPU^+(\A)$. Therefore,
\[
\psi \varphi = t^n \varphi^{\ast} \varphi = t^n.
\]

The last step is to show that $t$ is singular:

The divisors of $t$ are exactly the $p_M$ with $M \in X(\A^{\prime})$ (\autoref{lem:divisors_of_t}).  For $p_M, p_N$ we calculate
\[
p_M p_N  = (t \pi_M + \pi_{M^{\ast}})(t \pi_N + \pi_{N^{\ast}}) = t^2 \pi_M \pi_N + t(\pi_M \pi_{N^{\ast}} + \pi_{M^{\ast}} \pi_N) + \pi_{M^{\ast}} \pi_{N^{\ast}}.
\]
The $t^2$-coefficient is zero if and only if $M \bot N$. In this case, from \autoref{lem:multiplication_of_pms} follows $p_N p_M = p_{M \oplus N} = p_M \vee p_N$. So, $t$ is indeed singular.

Rump's \autoref{thm:which_groups_are_structure_groups} now tells us that $\left[ 1,t \right]$ - with the lattice-order inherited by $\PPU(\A)$ - is an OML under the complementation given by $\varphi^{\circledast} = \varphi^{-1}t$ and the natural embedding $\left[ 1,t \right] \hookrightarrow \PPU(\A)$ identifies the latter with its structure group (we use $\circledast$ for the orthocomplementation because $\ast$ \emph{will} cause confusion in what follows).

By \autoref{cor:gamma_is_an_isomorphism}, we already know that $\left[1,t \right] \cong X(\A^{\prime})$ with respect to the isomorphism $M \mapsto p_M$. The only left task is to show that the orthocomplementation is preserved, too:

Using paraunitarity,
\[
p_M^{\circledast} = p_M^{-1} t = p_M^{\ast} t = \pi_{M^{\ast}} + t \pi_{M} = p_{M^{\ast}},
\]
i.e. $\Gamma(M^{\ast}) = \Gamma(M)^{\circledast}$ meaning that $\Gamma$ is even an isomorphism of OMLs.

Therefore, $\PPU(\A) \cong G(\left[1,t \right]) \cong G(X(\A^{\prime}))$.
\end{proof}

\section{Possible generalizations} \label{sec:remarks}

Our definition of $\PPU(\A)$ is rather \glqq algebraic\grqq\, in nature, in that it is constructed as a ring of \emph{finite} Laurent series over $\A$.

Clearly, this restriction is necessary for $\PPU(\A)$ being a structure group. However, one might wonder if there is a slightly bigger group which is a structure group of $X(\A^{\prime})$ in another, more analytic, sense.

One more or less obvious extension is the following:

We first define
\[
l^1(\A) = \left\{ f: \Z \to \A: \sum_{i=-\infty}^{\infty} \Vert f(i) \Vert < \infty \right\}
\]
where $\Vert . \Vert$ is the operator norm. This becomes a Banach algebra under convolution and can be made a Banach-$\ast$-algebra by defining $f^{\ast}$ by $f^{\ast}(i) = (f(-i))^{\ast}$.

We can define an \emph{analytic pure paraunitary group} by
\[
\textnormal{pu}(\A) := \left\{ f \in l^1(\A): f^{\ast} f = 1 = f f^{\ast} \right\}.
\]
Specializing at $1$ still can be made sense in this framework, and so we could define an \emph{analytic} pure paraunitary group by
\[
\textnormal{ppu}(\A) := \left\{ f \in \textnormal{pu}(\A): \sum_{i=-\infty}^{\infty} f(i)=1 \right\}
\]
which is still a subgroup which is also right-ordered by the positive cone $\textnormal{ppu}^+(\A)$, consisting of the $f \in \textnormal{ppu}(\A)$ with $f(i) = 0$ for $i < 0$.

It is, however, not clear, if this is a lattice order. A meaningful definition of $\textnormal{ppu}(\A)$ should include this. Furthermore, a variation of \autoref{cor:omega_is_a_lattice_isomorphism}
should hold: mapping $\textnormal{ppu}(\A)$ to the subspace lattice of $\Hill$ by means of $f \mapsto f \Hilm$ should result in an order-isomorphism with a lattice of $\A^{\prime} \left[ \left[ t^{-1} \right] \right] $-invariant subspaces $M$ of $\Hill$ with some \glqq nice\grqq\, analytical properties.

One possible property could be that $M$ fulfils
\begin{align*}
\bigcup_{i=0}^{\infty} t^i M & = \Hill \quad \textnormal{and} \\
\bigcap_{i=0}^{\infty} t^{-i} M & = \{ 0 \}.
\end{align*}

However, it might be that a definition by means of $l^1(\A)$ is too naive and a more subtle kind of convergence is needed here.

In each case, an analytic pure paraunitary group - together with a compatible definition of \glqq analytic\grqq\, structure groups for general OMLs - could help understanding completion processes in lattice-ordered groups, may they be one- or two-sided.

\section*{Acknowledgements}

I am very grateful to Wolfgang Rump for suggesting the research topic and giving valuable advice on the presentation of this work.

\bibliographystyle{halpha}

\end{document}